\documentclass[11pt]{amsart}
\usepackage[a4paper,margin=3.5cm]{geometry}
\usepackage{amsmath,amssymb}
\usepackage[utf8]{inputenc}
\usepackage[T1]{fontenc}
\usepackage[english]{babel}
\usepackage{graphicx}
\usepackage[all,cmtip]{xy}

\newcommand{\Z}{\mathbb{Z}}

\theoremstyle{plain}
\newtheorem{theorem}{Theorem}[section]
\newtheorem{proposition}[theorem]{Proposition}
\newtheorem{corollary}[theorem]{Corollary}
\newtheorem{lemma}[theorem]{Lemma}

\theoremstyle{remark}

\theoremstyle{definition}

\title{Some results on extension of maps and applications}

\author[C.\ Biasi]{Carlos Biasi}
\address{Departamento de Matem\'atica\\
	Instituto de Ci\^encias Matem\'aticas e de Computa\c c\~ao\\
	S\~ao Paulo University - C\^ampus de S\~ao Carlos \\
	13560-970, S\~ao Carlos, SP, Brazil}
\email{biasi@icmc.usp.br}

\author[A.\ K.\ M.\ Libardi]{Alice K.\ M.\ Libardi}
\address{Departamento de Matem\'{a}tica\\
	S\~ao Paulo State University (Unesp), Institute of Geosciences and Exact Sciences, Rio Claro\\
	Bela Vista\\
	13506-700, Rio Claro, SP, Brazil}
\email{alicekml@rc.unesp.br}

\author[T.\ de Melo]{Thiago de Melo} 
\address{Departamento de Matem\'{a}tica\\
	S\~ao Paulo State University (Unesp), Institute of Geosciences and Exact Sciences, Rio Claro\\
	Bela Vista\\
	13506-700, Rio Claro, SP, Brazil}
\email{tmelo@rc.unesp.br}

\author[E.\ L.\ dos Santos]{Edivaldo L.\ dos Santos}
\address{Departamento de Matem\'{a}tica \\
	Federal University of S\~ao Carlos \\
	Rodovia Washington Luiz. km 235 \\
	S\~ao Carlos, SP, Brazil}
\email{edivaldo@dm.ufscar.br}


\thanks{This work is partially supported by the Projeto Tem\'atico: Topologia Alg\'ebrica, Geom\'etrica e Diferencial, FAPESP Process Number 2016/24707-4}

\begin{document}
\baselineskip=1.5em

\maketitle
\hfill \centerline{\it \hspace{1,5cm} Dedicated to Professor Gilberto Loibel, in memorian.}

\begin{abstract}
 This paper concerns  extension of maps using obstruction theory under a non classical viewpoint. It is given a classification of homotopy classes
 of maps and as an application it is presented a simple proof of a theorem by Adachi about equivalence of vector bundles. Also  
 it is proved that, under certain conditions,  two embeddings are homotopic up to surgery if and only if the respective normal bundles are $SO$-equivalent.
\end{abstract}

\section{Introduction}


Two embeddings $f,g\colon M\to N$ between manifolds are homotopic up to surgery on $N$ if it is possible to make a finite number of surgeries on $N$ outside of the images of $f$ and $g$ obtaining a new manifold $N'$
and maps $f',g'\colon M\to N'$ such that $f'$ and $g'$ are homotopic. 

Consider the case where $M$ and $N$ are closed orientable $C^\infty$ manifolds of dimensions $m$ and $n$, respectively, and $m\leq \frac{n-2}{2}$. Then we ask on 
which conditions $f$ and $g$ are homotopic up to surgery
on $N$? In \cite{loibel}, the authors showed that if  $M=\mathbb{S}^m$ or $N$ is a $\pi$-manifold then $f$ and $g$ are homotopic up to surgery on $N$.


In this paper we prove the following theorem.

\begin{theorem}\label{teo.3.6}If $M$ and $N$ are two closed orientable manifolds of dimensions $m$ and $n$, respectively, with $m\leq \frac{n-2}{2}$, such that: \begin{enumerate}
\item[a)] $H^{4k}(M;\Z)$ are free groups, for all $k\geq 1$,
\item[b)] $H^{8k+1}(M;\Z_2)=H^{8k+2}(M;\Z_2)=0$, for all $k\geq 1$,
\end{enumerate}
then two embeddings $f,g\colon M\to N$ are homotopic up to surgery on $N$ if and only if the normal bundles $\nu_f$ and $\nu_g$ are equivalent as orientable vector bundles.
\end{theorem}
 
Also, as an application of this theory we give a 
simple proof of a nice  result of Adachi \cite{adachi}, about equivalence of vector bundles.

This paper is organized as follows:  in Section~\ref{sec.2}, results on  homotopy of maps are considered and in Section~\ref{sec.3},
applications of the theory are given.

%
%
%

\section{Homotopy of maps}\label{sec.2}

Let $X$ be a CW-complex and $Y$ an $n$-simple CW complex, for some $n\geq 1$. For an abelian group $G$, let $\mathcal{K}=K(G,n)$ be an Eilenberg--MacLane space and $i_n\in H^n(\mathcal{K};G)$ an $n$-characteristic element.

Given $u\in H^n(Y;G)$, by \cite[8.1.10]{spanier} there exists a map $\varphi \colon Y\to \mathcal{K}$ such that $\varphi^*(i_n)=u$. The induced homomorphism
$\varphi_*\colon \pi_n(Y)\to \pi_n(\mathcal{K})$ gives a  coefficient  homomorphism \begin{equation}\varphi_u^p\colon H^p(X; \pi_n(Y))\to H^p(X;\pi_n(\mathcal{K})),  \text{ for all $p$.} \label{eq1} \end{equation}

\begin{theorem}Let $f,g\colon X^{(n+1)}\to Y$ be homotopic maps over $X^{(n-1)}$ satisfying $f^*(u)=g^*(u)$. If $\varphi_u^n$  is a monomorphism, then $f$ and $g$ are homotopic maps over $X^{(n)}$.\label{teo.2.1}
\end{theorem}

\begin{proof}Let $H\colon X^{(n-1)}\times I\to Y$ be a homotopy between the restricted maps  $f_{n-1}=f|X^{(n-1)}$ and $g_{n-1}=g|X^{(n-1)}$, and consider $d(f_n,H,g_n)\in H^n(X;\pi_n(Y))$ 
the obstruction to the extension of $H$ to a homotopy between $f_n=f|X^{(n)}$ and $g_n=g|X^{(n)}$.

Since $\varphi_u^n(d(f_n,H,g_n))=d(\varphi\circ f_n,\varphi\circ H,\varphi\circ g_n)$ and $f^*(u)=g^*(u)$, then
$(\varphi\circ f)^*(i_n)=(\varphi\circ g)^*(i_n)$. Hence, $d(f_n,H,g_n)=0$ and so the result follows.
\end{proof}

\begin{theorem} Given a map $f\colon X^{(n+1)}\to Y$ and $\alpha\in H^n(X^{(n+1)};G)$, there exists a map
$g\colon X^{(n+1)}\to Y$ such that $g^*(u)=\alpha$ and $g|X^{(n-1)}=f|X^{(n-1)}$ if and only if $\alpha-f^*(u)$ belongs
to the image of $\varphi_u^n$.\label{teo.2.2}
\end{theorem}


\begin{proof}
Let $c\colon X^{(n+1)}\to Y$ be a constant map. Suppose that there exists $g\colon X^{(n+1)}\to Y$ as stated in Theorem~\ref{teo.2.2}. Then, one has $\varphi_u^n(d(f_n,g_n))=d(c,\varphi\circ g_n)-d(c,\varphi\circ f_n)=\alpha- f^*(u)$.

Conversely if $\alpha-f^*(u)$ belongs to the image of $\varphi_u^n$, let $\beta\in H^n(X;\pi_n(Y))$ be such that
$\varphi_u^n(\beta)=\alpha -f^*(u)$.

It follows then the existence of $g\colon X^{(n+1)}\to Y$ such that $d(f_n,g_n)=\beta$ and
$g^*(u)=\varphi_u^n(d(f_n,g_n))+d(c,\varphi\circ f_n)=\alpha$.
\end{proof}

\begin{theorem}Consider $f\colon X^{(n)}\to Y$ and $\alpha\in H^n(X^{(n+1)};G)$. If $\varphi_u^n$ is an
epimorphism and $\varphi_u^{n+1}$ is a monomorphism, then there exists $g\colon X^{(n+1)}\to Y$ such that $g|X^{(n-1)}=f|X^{(n-1)}$
and $g^*(u)=\alpha$.\label{teo.2.3}
\end{theorem}

\begin{proof}
Let $o(f)\in H^{n+1}(X;\pi_n(Y))$ be the obstruction to the extension of $f_{n-1}$ to $(n+1)$-skeleton. 
Since $\varphi_u^{n+1}(o(f))=o(\varphi\circ f)=0$ and $\varphi_u^{n+1}$ is a monomorphism we have that $o(f)=0$. 
Then, there exists an extension of $f_{n-1}$ to $X^{(n+1)}$ and so, applying Theorem~\ref{teo.2.2} to it, the proof is finished.
\end{proof}

\begin{proposition}\label{prop.2.4} Assuming that $Y$ is $n$-simple for all $n\leq \dim X$, $H^i(X;\pi_i(Y))=0$ 
for $i\neq n$, $H^{i+1}(X;\pi_i(Y))=0$ for $i>n$, and $\varphi_u^n$ is a monomorphism, then the map 
$E\colon [X,Y]\to H^n(X;G)$ defined by $E([f])=f^*(u)$ is injective and $\operatorname{im} E=\operatorname{im} \varphi_{u}^{n}\approx H^n(X;\pi_n(Y))$.
\end{proposition}

\begin{proof}
Since $H^i(X;\pi_i(Y))=0$, for $i<n$, any two maps $f,g\colon X\to Y$ are homotopic over the $(n-1)$-skeleton. 
If $f^*(u)=g^*(u)$ and $\varphi_u^n$ is a monomorphism it follows from Theorem~\ref{teo.2.1} that $f$ and $g$ are
homotopic over the $n$-skeleton. Since $H^i(X;\pi_i(Y))=0$ for $i>n$, then $f$ and $g$ are homotopic and so $E$ is injective.

Given $f\colon X\to Y$ then  $f$ is homotopic to a constant map over the $(n-1)$-skeleton. It follows from
Theorem~\ref{teo.2.2} that $E([f])=f^*(u)$ belongs to the image of~$\varphi_u^n$.

If $\alpha\in H^n(X;G)=H^n(X^{(n+1)};G)$, by Theorem~\ref{teo.2.2} there exists $h\colon X^{(n+1)}\to Y$ 
such that $h^*(u)=\alpha$ and $h$ is constant over the $(n-1)$-skeleton. Since $H^{i+1}(X;\pi_i(Y))=0$ for $i>n$,
there exists $f\colon X\to Y$ with $f^*(u)=\alpha$. Hence $\operatorname{im} E=\operatorname{im} \varphi_u^n$.
\end{proof}

\begin{theorem}\label{teo.2.5} Suppose that $Y$ is $n$-simple, for all $n\leq \dim X$ and $J=\{n \mathrel{:} H^n(X;\pi_n(Y))\neq 0 \}\not=\emptyset$. For each $n\geq 1$, let $G_n$ be an abelian group,  $u_n\in H^n(Y;G_n)$ and $\varphi_{u_{n}}^n\colon H^n(X; \pi_n(Y))\to H^n(X;G_n)$ as  in {\rm (\ref{eq1})}. If $\varphi_{u_n}^{n}$ is a monomorphism, then the function
\[E\colon [X,Y]\to \prod_{n\in J} H^n(X;G_n)\] given by $E([f])=\prod_{n\in J}f^*(u_n)$ is injective
\textup{(}where $\prod$ denotes the Cartesian product\textup{)}.
\end{theorem}

\begin{proof}
Let $p$ be the least positive integer such that $H^p(X;\pi_p(Y))\neq 0$. Since $H^i(X;\pi_i(Y))=0$, $i<p$, 
any two maps $f,g\colon X\to Y$ are homotopic over $X^{(p-1)}$.

Suppose  $f^*(u_p)=g^*(u_p)$. Since $\varphi_{u_{p}}^p$ is a monomorphism, it follows
from Theorem~\ref{teo.2.1} that $f$ and $g$ are homotopic over $X^{(p)}$  and so applying 
Theorem~\ref{teo.2.1} successively, the result follows.
\end{proof}

Now, we present some interesting consequences of Theorem~\ref{teo.2.5}.

Let $\mathcal{K}=K(H_n(Y),n)$ be an Eilenberg--MacLane space. Then, the map $\varphi\colon Y\to \mathcal{K}$ such that $\varphi^*(i_n)=u_n$ 
induces the Hurewicz homomorphism
\[\varphi_{*}=\mathfrak{X}(Y)\colon \pi_n(Y)\to \pi_n(\mathcal{K})=H_n(Y).\]

Let us consider 
\[\mathfrak{X}_n^p(Y) =\varphi_{u_{n}}^p \colon H^p(X;\pi_n(Y))\to H^p(X;H_n(Y))\] the coefficient homomorphism as defined in (\ref{eq1}).

Let $u_n\in H^n(Y;H_n(Y))$ be an element such that $\langle u_n,x\rangle=x$, for all $x\in H_n(Y)$.

\begin{theorem}\label{teo.2.6}If $\mathfrak{X}_n^n(Y)$ is a monomorphism for all $n\in J$ then the function $E$ is injective.
\end{theorem}

\begin{proof}
Let $\mathcal{K}=K(H_n(Y),n)$ and observe that the  commutative diagram
\[
\xymatrix@C=1.5cm{
\pi_n(Y)\ar[r]^{\mathfrak{X}(Y)} \ar[d]_{\varphi_*} & H_n(Y)\ar[d]^{\varphi_*} \\
\pi_n(\mathcal{K}) \ar[r]_{\mathfrak{X}(\mathcal{K})} & H_n(\mathcal{K})
}
\] induces for all $n$ the commutative diagram
\[
\xymatrix@C=1.5cm{
H^n(X;\pi_n(Y))\ar[r]^{\mathfrak{X}_n^n(Y)} \ar[d]_{\varphi_{u_n}^n} & H^n(X;H_n(Y))\ar[d]^{\mathrm{id}^n} \\
H^n(X;\pi_n(\mathcal{K})) \ar[r]_{\mathfrak{X}_n^n(\mathcal{K})} & H^n(X;H_n(\mathcal{K}))\rlap{.}
}
\]
Then,  the result follows from Theorem~\ref{teo.2.5}.
\end{proof}

\begin{corollary}\label{cor.2.7}Suppose $\mathfrak{X}_n^n(Y)$ is a monomorphism  and $H_{n-1}(X)$ a free group, for all
$n\in J$. Then, two maps $f,g\colon X \to Y$ are homotopic if and only if for each $n\in J$ the induced homomorphisms
$f_*,g_*\colon H_n(X)\to H_n(Y)$ are equal.
\end{corollary}

If $Y$ is $(n-1)$-connected ($\pi_1(Y)$ abelian if $n=1$) and $u\in H^n(Y;\pi_n(Y))$ is the element corresponding to
the inverse of the Hurewicz homomorphism \[\mathfrak{X}(Y)\colon \pi_n(Y)\to H_n(Y)\] then
$\varphi_*\colon \pi_n(Y)\to \pi_n(\mathcal{K})$ is the identity map, where $\varphi\colon Y\to \mathcal{K}$ is such that $\varphi^*(i_n)=u$. It follows that $\varphi_u^p$ is an isomorphism for all $X$ and $p$. This is true if $u$ is any $n$-characteristic element.

Let $Y$ and $Z$ be topological spaces and $\varphi\colon Y\to Z$ a map.
The following commutative diagram
\[
\xymatrix@C=1.5cm@R=1.25cm{
\pi_n(Y)\ar[r]^{\varphi_*} \ar[d]_{\mathfrak{X}(Y)} & \pi_n(Z)\ar[d]^{\mathfrak{X}(Z)} \\
H_n(Y) \ar[r]_{\varphi_*} & H_n(Z)
}
\] induces for all topological space $X$, the commutative diagram
\[
\xymatrix@C=1.5cm@R=1.25cm{
H^p(X;\pi_n(Y))\ar[r]^{\varphi^p} \ar[d]_{\mathfrak{X}_n^p(Y)} & H^p(X;\pi_n(Z))\ar[d]^{\mathfrak{X}_n^p(Z)} \\
H^p(X;H_n(Y)) \ar[r]_{\varphi^p} & H^p(X;H_n(Z))\rlap{.}
}
\]

\begin{theorem}\label{teo.3.5} Let $X$ and $Y$ be CW-complexes and $Y$  $i$-simple, for all $i\leq \dim X$. If $\mathfrak{X}_i^i(Z)\circ \varphi^i\colon H^i(X; \pi_i(Y))\to H^i(X;H_i(Z))$ are monomorphisms, then $\varphi$ induces an injective map $\varphi_\#\colon [X,Y]\to [X,Z]$ defined by $\varphi_\#([f])=[\varphi\circ f]$.
\end{theorem}

\begin{proof}Consider $\mathcal{K}=K(H_i(Z),i)$, $u_Z^i\in H^i(Z;H_i(Z))$ for all $i$ and $E'\colon [X,Z]\to \prod_i H^i(X;H_i(Z))$ 
defined by $E'([\varphi\circ f])=\prod_i(\varphi\circ f)^*(u_Z^i)$.

Let us define $E\colon [X,Y]\to \prod_i H^i(X;H_i(Z))$ by $E([f])=\prod_i f^*(u_Y^i)$ where $u_Y^i=\varphi^*(u_Z^i)$. 
So we have the commutative diagram
\[
\xymatrix@C=1.7cm{
[X,Y]\ar[r]^-{E} \ar[d]_{\varphi_\#}  & \prod_i H^i(X;H_i(Z)) \\
  [X,Z] \ar[ru]_{E'} &
}
\] and since $\mathfrak{X}_i^i(Z)\circ \varphi^i$ are monomorphisms for all $i$, it follows from Theorem~\ref{teo.2.5}
that $E$ is injective and so, $\varphi_\#$ is injective as well.
\end{proof}

\section{Applications}\label{sec.3}

Given an $O$-stable vector bundle $\xi\to B$ where $O$ is the stable orthogonal group, let $w_i(\xi)\in H^i(B;\Z_2)$ and $p_k(\xi)\in H^{4k}(B;\Z)$ be  the Stiefel--Whitney and Pontrjagin classes of $\xi$, respectively.

Let $\gamma\to BO$ be the $O$-universal bundle and let $w_i=w_i(\gamma)$ and $p_k=p_k(\gamma)$ be  the universal
Stiefel--Whitney and Pontrjagin classes, respectively.

Write $\pi_n=\pi_n(BO)$, $\mathcal{K}_n=K(\pi_n,n)$ and let  $i_n\in H^n(\mathcal{K}_n;\pi_n)$ be the element corresponding to the inverse of Hurewicz homomorphism.

Also, consider the maps $\varphi_n\colon BO\to \mathcal{K}_n$ for $n\geq 1$ satisfying \[\varphi_n^*(i_n)=\begin{cases}
w_n, & n=1,2,\\ p_k, & n=4k,\ k\geq 1,
\end{cases} \] and denote by $\varphi_{\#} \colon \pi_n(BO)\to \pi_n(\mathcal{K}_n)$ the homomorphism induced by $\varphi_n$, which induces the coefficient homomorphism \[\varphi^{n}_{n}\colon H^n(X;\pi_n(BO))\to H^n(X;\pi_n), \] for any topological space $X$.

\begin{lemma}\label{lema.3.1}Let $X$ be a topological space satisfying: \begin{enumerate}
\item\label{lema.3.1.a} if $x\in H^{4k}(X;\Z)$, $x\neq 0$, and $k\geq 1$,  then $(2k-1)!a_kx\neq 0$,
where \[a_k=\begin{cases}1, & \text{for $k$ even,}\\ 2, & \text{for $k$ odd.}\end{cases} \]
\item\label{lema.3.1.b} $H^{8j+1}(X;\Z_2)=H^{8j+2}(X;\Z_2)=0$, for each $j\geq 1$.
\end{enumerate}

Under these conditions, the  coefficient homomorphism  $\varphi^{n}_{n}$ is a monomorphism for $n=1,2$  or $n=4k$, with $k\geq 1$.
\end{lemma}

\begin{proof}
For $n=1,2$ the proof is straightforward since $\pi_n(BO)=\Z_2$.

For $n=4k$, with $k\geq 1$, the homomorphism $\varphi\colon \pi_n(BO)\to \Z$ is given by $\varphi(\alpha)=(2k-1)!a_k$, where $\alpha$ is the generator of $\pi_n(BO)$ (see \cite{adachi}).

Consider the long exact sequence \[\xymatrix@C=.6cm{\cdots\ar[r] & H^{n-1}(X;H) \ar[r]^\beta & H^n(X;\Z)\ar[r]^{\varphi^n} & H^n(X;\Z)\ar[r] & H^n(X;H) \ar[r] & \cdots} \] where $H=\Z/\operatorname{im} \varphi$ and $\beta$ is the Bockstein operator. We see that there are no nonzero elements of $H^n(X;\Z)$
belonging to the image of $\beta$, because if there is such an element we would have $(2k-1)!a_kx=0$.

Then, $\varphi^n$ are monomorphisms, for $n=4k$ with $k\geq 1$.
\end{proof}

\begin{theorem}[Adachi]\label{teo.3.2}Let $X$ be a CW complex satisfying conditions \ref{lema.3.1.a} 
and \ref{lema.3.1.b} of Lemma~\ref{lema.3.1} and let $\xi_1$ and $\xi_2$ be two $O$-stable vector bundles over $X$. Then, $\xi_1$ and $\xi_2$ are equivalent if and only if $w_1(\xi_1)=w_1(\xi_2)$, $w_2(\xi_1)=w_2(\xi_2)$ and
$p_k(\xi_1)=p_k(\xi_2)$, for every $k\geq 1$.
\end{theorem}

\begin{proof}
Let $f_1,f_2\colon X\to BO$ be the classifying maps for $\xi_1$ and $\xi_2$. Thus,
\begin{align*}
f^*_1(w_i)&=w_i(\xi_1)=w_i(\xi_2)=f_2^*(w_i),\ i=1,2,\\ f^*_1(p_k)&=p_k(\xi_1)=p_k(\xi_2)=f^*_2(p_k),\ k\geq 1,
\end{align*}
and since $H^n(X;\pi_n)=0$ for $n\neq 1,2$ and $n \neq 4k$, $ k \geq 1$ and $\varphi^n$ is a monomorphism in these dimensions,  it follows from Theorem~\ref{teo.2.5} that $f_1$ is homotopic to $f_2$.
\end{proof}

Let $f\colon \mathbb{S}^p\to N^n$ be an embedding with trivial normal bundle. Consider $N'$ the manifold obtained by surgery on
$N$ along the embedding $f$. In this case, we say that a surgery of type $(p+1,n-p)$ was done on $N'$.

From \cite{milnor}, one has the following:

\begin{proposition}\label{prop.3.3} If $N$ is a compact orientable manifold then it is possible to obtain $N'$ by surgeries of type $(p+1,n-p)$ on $N$, with $p\leq \frac{n-2}{2}$, such that the map $\varphi_i\colon \pi_i(N')\to \pi_i(BSO(n))$ induced by the classifying map for the tangent bundle of $N'$ is injective for $i\leq \frac{n-2}{2}$.
\end{proposition}


Let now $\varphi_{N\#}\colon[M,N]\to [M,BSO(n)]$ be the map defined by $\varphi_{N\#}([f])=[\varphi_N\circ f]$, 
where $\varphi_N\colon N\to BSO(n)$ is the classifying map for the tangent bundle of $N$.


Next, we prove the main result of the paper.



\begin{proof}[\proofname\ of Theorem~\ref{teo.3.6}]
Let $N'$ be the manifold obtained by surgery on $N$ (of type $(p+1,n-p)$ with $p\leq \frac{n-2}{2}$)
outside of the images of $f$ and $g$, such that $\varphi_i\colon\pi_i(N')\to \pi_i(BSO(n))$, the induced of the classifying map of $N'$, is injective for $i\leq \frac{n-2}{2}$ (see Proposition~\ref{prop.3.3}). Moreover, since $N'$ is $1$-parallelizable ($N'$ is  orientable), by \cite[Theorem~3]{kervaire-milnor} we get that $N'$ is $1$-connected, and consequently $N'$ is $i$-simple for all $i$.

Now, we show that the coefficient homomorphism \[\mathfrak{X}^i_i(BSO(n))\circ\varphi^i\colon H^i(M;\pi_i(N'))\to H^i(M;H_i(BSO(n))), \] induced by $\varphi_i$ and the Hurewicz homomorphism $\mathfrak{X}_i(BSO(n))$, is a monomorphism for $i\leq \dim M = m\leq  \frac{n-2}{2}$ (cf.\ \cite[Theorem~3.1]{kervaire-milnor}).

If $i=8k+1$ or $8k+2$, for $k\geq 1$, with $i\leq \frac{n-2}{2}$, then $\pi_i(BSO(n))=\Z_2$ and since $\varphi_i$ is injective it follows that $\pi_i(N')$ is either the trivial group or $\Z_2$. If $\pi_i(N')=\Z_2$, since $H^{8k+1}(M;\Z_2)=H^{8k+2}(M;\Z_2)=0$ for $k\geq 1$ by assumption, it follows that $\mathfrak{X}^i_i(BSO(n))\circ\varphi^i$ is a monomorphism.

For $n\geq 6$, observe that $\mathfrak{X}_2(BSO(n))$ and $\varphi_2$ are isomorphisms because  $\pi_1(BSO(n))=0$ and $\pi_2(BSO(n))=\Z_2$.

Consider $h\colon \mathbb{S}^{4k}\to BSO(n)$ a generator of $\pi_{4k}(BSO(n))$ and $p_k$ a Pontrjagin class of $h^*(\gamma_n)$. Then, we obtain that $h^*(p_k)=(-1)^{k+1}(2k-1)!a_ks$, where $s$ is a generator of $H^{4k}(\mathbb{S}^{4k})$ (see \cite{adachi}) and $a_k$ is defined in Lemma~\ref{lema.3.1}(\ref{lema.3.1.a}).

Further, $h_*([\mathbb{S}^{4k}])\neq 0$, where $[\mathbb{S}^{4k}]\in H_{4k}(\mathbb{S}^{4k})$ is the fundamental class of $\mathbb{S}^{4k}$. Thus, if $\beta$ is a non zero element of $\pi_{4k}(BSO(n))$ then $\mathfrak{X}_{4k}(\beta)$ is a multiple of $h_*([\mathbb{S}^{4k}])$, from which it follows that the Hurewicz homomorphism $$\mathfrak{X}_{4k}\colon \pi_{4k}(BSO(n))\to H_{4k}(BSO(n))$$ is a monomorphism, for $k\geq 1$.

Next, if $i=4k$ with $i\leq \frac{n-2}{2}$, we have $\pi_i(BSO(n))=\mathbb{Z}$, $H_{4k-1}(M)$ is torsion free and then $\varphi^{4k}$  and $\mathfrak{X}^{4k}$ are monomorphisms. So, we conclude that $\mathfrak{X}^{4k}\circ \varphi^{4k}$ are monomorphisms also, for $4k\leq \frac{n-2}{2}$. For other $k$, the groups $\pi_i(BSO(n))$ are trivial.

By Theorem~\ref{teo.3.5}, we conclude that $\varphi\colon N'\to BSO(n)$ induces the injective map $$\varphi_\#\colon [M,N']\to [M,BSO(n)].$$ 

In addition, since $\varphi_N\circ f\colon M\to BSO(n)$ classifies $f^*(\tau N)$,
the diagram \[ \xymatrix@C=1.5cm{ M \ar[r]^{f,g}\ar[rd]_{f',g'} & N \ar[r]^-{\varphi_N} & BSO(n)\\ 
 & N' \ar[ru]_{\varphi}
} \] guarantees that $\nu_f \simeq \nu_g  \Leftrightarrow \varphi_N\circ f \simeq \varphi_N \circ g \Leftrightarrow \varphi\circ f' \simeq \varphi \circ g' \Leftrightarrow f'\simeq g'$.
\end{proof}

\subsection*{Acknowledgment}  The second named author would like to thank the hospitality of the Department of Mathematics of the Federal University of S\~ao Carlos (UFSCar).

%
%
\end{document}